\documentclass[11pt,a4paper]{amsart}
\usepackage{amssymb, amstext, amscd, amsmath}
\usepackage{url}
\usepackage{amsfonts}
\usepackage{lscape}
\usepackage{amsthm}
\usepackage{amsfonts}
\usepackage{array}
\usepackage{eucal}
\usepackage{latexsym}
\usepackage{mathrsfs}
\usepackage{textcomp}
\usepackage{verbatim}
\usepackage{tikz}
\usepackage{kbordermatrix}

\newtheorem{thm}{Theorem}[section]

\newtheorem{con}[thm]{Conjecture}

\newtheorem{lem}[thm]{Lemma}

\numberwithin{equation}{section}

\newcommand{\bQ}{{\mathbb{Q}}}
\newcommand{\R}{{\mathbb{R}}}

\newcommand{\rank}{\operatorname{rank}}

\newcommand{\dmt}{{[d-t]}}

\tikzset{nodeblack/.style={circle,draw=black,fill=black!30,inner sep=1.2pt}}
\tikzset{every loop/.style={}}
\begin{document}

\title{An improved bound for the rigidity of linearly constrained frameworks}

\author[Bill Jackson]{Bill Jackson}
\address{School of Mathematical Sciences, Queen Mary
University of London, Mile End Road, London\\ E1 4NS \\ U.K. }
\email{b.jackson@qmul.ac.uk}
\author[Anthony Nixon]{Anthony Nixon}
\address{Department of Mathematics and Statistics\\ Lancaster University\\ Lancaster \\
LA1 4YF \\ U.K. }
\email{a.nixon@lancaster.ac.uk}
\author[Shin-ichi Tanigawa]{Shin-ichi Tanigawa}
\address{Department of Mathematical Informatics\\ Graduate School of Information
Science and Technology\\ University of Tokyo\\ 7-3-1 Hongo\\ Bunkyo-ku\\
113-8656\\ Tokyo \\ Japan}
\email{tanigawa@mist.i.u-tokyo.ac.jp}
\date{\today}

\begin{abstract}
We consider the problem of characterising the generic rigidity of bar-joint frameworks in $\R^d$ in which each vertex is constrained to lie in a given affine subspace. The special case when $d=2$ was previously solved by I. Streinu and L. Theran in 2010 and the case when 
each vertex is constrained to lie in an affine subspace of dimension $t$, 
and $d\geq t(t-1)$ was solved by Cruickshank, Guler and the first two authors in 2019.
We extend the latter result by showing that the given characterisation holds whenever $d\geq 2t$.
\end{abstract}


\maketitle

\section{Introduction}\label{introduction}

A (bar-joint) framework $(G,p)$ in $\mathbb{R}^d$ is the combination of a simple graph $G=(V,E)$ and a  realisation $p:V\rightarrow \mathbb{R}^d$. The framework $(G,p)$ is \emph{rigid} if every edge-length preserving continuous motion of the vertices arises as a congruence of $\mathbb{R}^d$.

It is NP-hard to determine whether a given framework is rigid \cite{Abb}, but this problem becomes more tractable when one considers the \emph{generic} behaviour. It is known that the rigidity of a generic framework $(G,p)$ in $\mathbb{R}^d$ depends only on the underlying graph $G$, see \cite{AR}. We say that $G$ is \emph{rigid} in $\mathbb{R}^d$ if some (and hence every) generic realisation of $G$ in $\mathbb{R}^d$ is rigid. The problem of characterising graphs which are rigid in $\R^d$ has been solved for $d=1,2$ but is open for all $d\geq 3$.

We will consider the problem of characterising the generic rigidity of bar-joint frameworks in $\R^d$ with additional constraints that  require some  vertices to lie in given affine subspaces. We model the underlying incidence structure of such a framework as  a {\em looped simple graph} $G=(V,E,L)$ where the vertex set $V$ represents the joints, the edge set $E$ represents the distance constraints between pairs of distinct vertices and the loop set $L$ represents the subspace constraints on individual vertices.  We will distinguish between edges and loops throughout the paper, an edge will always have two distinct end-vertices and a loop will always have two identical end-vertices.

Motivated by potential applications in sensor network localisation and in mechanical engineering, rigidity has already been considered for bar-joint frameworks with various kinds of additional constraints \cite{CGJN,KatTan,SSW,ST}. Following \cite{CGJN}, we define a {\em linearly constrained framework in $\R^d$}  to be a triple $(G,
p, q)$ where $G=(V,E,L)$ is a looped simple graph,
$p:V\to \R^d$ and
$q:L\to \R^d$. For $v_i\in V$ and $\ell_j\in L$ we put $p(v_i)=p_i$ and
$q(\ell_j)=q_j$.
The framework $(G,p)$ is {\em generic} if $(p, q)$ is algebraically independent over
$\mathbb{Q}$.

An {\em infinitesimal motion} of $(G, p, q)$ is a map $\dot
p:V\to \R^d$ satisfying the system of linear equations:
\begin{eqnarray}
\label{eqn1} (p_i-p_j)\cdot (\dot p_i-\dot p_j)&=&0 \mbox{ for all $v_iv_j \in E$}\\
\label{eqn2} q_j\cdot \dot p_i&=&0 \mbox{ for all incident pairs $v_i\in V$ and
$\ell_j \in L$.}
\end{eqnarray}
The second constraint implies that  the infinitesimal velocity of each
$v_i\in V$ is constrained to lie on the hyperplane through $p_i$ with normal vector $q_j$
for each loop $\ell_j$ incident to $v_i$.

The {\em rigidity matrix $R (G, p, q)$} of the linearly constrained framework $(G, p, q)$ is  the
matrix of coefficients of this system of equations for the unknowns
$\dot p$. Thus $R (G, p, q)$ is a $(|E|+|L|)\times d|V|$ matrix, in
which: the row indexed  by an edge $v_iv_j\in E$ has $p(u)-p(v)$ and
$p(v)-p(u)$ in the $d$ columns indexed by $v_i$ and $v_j$,
respectively and zeros elsewhere; the row indexed  by a loop
$\ell_j=v_iv_i\in L$ has $q_j$  in the $d$ columns indexed by $v_i$ and
zeros elsewhere. The $|E|\times d|V|$ sub-matrix consisting of the rows indexed by $E$ is the {\em bar-joint rigidity matrix} $R(G-L,p)$ of the bar-joint framework $(G-L,p)$.

The framework $(G, p, q)$ is {\em infinitesimally rigid} if its only
infinitesimal motion is $\dot p =0$, or equivalently if $\rank R(G,
p, q) = d|V|$. We say that the looped simple graph $G$ is  {\em rigid} in $\R^d$
if $\rank R(G, p, q) = d|V|$ for some realisation $(G,p, q)$ in
$\R^d$, or equivalently if $\rank R(G, p, q) = d|V|$ for all {\em
generic} realisations $(G,p,q)$ i.e.\  all realisations for which
$(p,q)$ is algebraically independent over $\bQ$.
Streinu and Theran \cite{ST} gave a complete characterisation of looped simple graphs which are rigid in  $\mathbb{R}^2$. 
Cruickshank et al.~\cite{CGJN} extended their characterisation to higher dimensions for graphs in which each vertex is incident to sufficiently many loops.
We need to introduce some terminology to describe this result. 

Given a looped simple graph  $G=(V,E,L)$ and $X\subseteq V$ let $i(X)$ denote the number of edges and loops in the subgraph of $G$ induced by $X$. We say that $G$ is {\em $k$-sparse} for some integer $k\geq 1$, if $i(X)\leq k|X|$ for all $X\subseteq V$ and that $G$ is {\em $k$-tight} if it is a $k$-sparse graph with $|E\cup L|=k|V|$.  
Let $G^{[k]}$ denote the graph obtained from $G$ by adding $k$ new loops at every vertex.
The following conjecture is posed in \cite{CGJN}.

\begin{con}[\cite{CGJN}] \label{conj:main}
Suppose $G$ is a looped simple graph and $d,t$ are positive integers with $d\geq 2t$. Then $G^\dmt$ can be realised as an infinitesimally rigid
linearly constrained framework in $\R^d$ if and
only if $G$ has a $t$-tight looped simple spanning  subgraph.
\end{con}

The main result of \cite{CGJN}  verifies Conjecture~\ref{conj:main} in the case when $d\geq \max\{2t,t(t-1)\}$. 
Our main result, Theorem~\ref{thm:main} below, verifies Conjecture~\ref{conj:main} completely and, in addition, extends the characterisation to the case when 
$d=2t-1$.

\section{Pinned independence}

Let $G=(V,E)$ be a simple graph and $P\subseteq V$. We will consider infinitesimal motions $\dot p$ of a $d$-dimensional bar-joint framework $(G,p)$ in which the vertices in $P$ are pinned i.e. $\dot p(v)=0$ for all $v\in P$.
Let $R^{pin}(G,P,p)$ denote the submatrix obtained from the rigidity matrix $R(G,p)$ by deleting the $d$-tuples of columns corresponding to vertices of $P$. We say that $(G,P)$ is \emph{pinned independent} in $\R^d$ if the rows of $R^{pin}(G,P,p)$ are linearly independent for any generic $p$. 

A graph $G'$ is said to be obtained from another graph $G$ by a \emph{0-extension} if $G=G'-v$ for a vertex  $v\in V(G')$ with $d_{G'}(v)=d$, or a \emph{1-extension} if $G=G'-v+xy$ for a vertex $v\in V(G')$ with $d_{G'}(v)=d+1$ and $x,y\in N(v)$.
We can use standard proof techniques to show that 0-extension and 1-extension preserve pinned independence, see for example \cite{SSW}.

\begin{lem}\label{lem:01ext}
Let $(G,P)$ be pinned independent and let $(G',P)$ be obtained from $G$ by a 0-extension or a 1-extension. Then $(G',P)$ is pinned independent.
\end{lem}

We can use this lemma to obtain a sufficient condition for pinned independence. 

\begin{lem}\label{lem:pin}
Let $G=(V,E)$ be a simple graph with $P\subseteq V$ and $d\geq 2$ be an integer. Construct a looped simple graph $G'$ from $G$ by adding $d$ loops to each vertex of $P$ and 
$\lfloor\frac{d}{2}\rfloor$ 
loops to each vertex of $V-P$. 
Suppose that $G'$ is $d$-sparse.
Suppose further that $G$ contains no subgraph isomorphic to $K_{d+2}$ when $d$ is odd. Then $(G,P)$ is pinned independent in $\mathbb{R}^d$.
\end{lem}

\begin{proof}
We prove the lemma by induction on $|V|$.
The conclusion is trivial if $V=P$ or $|V|=1$ so we may suppose not. Moreover we may assume $G$ is connected since the lemma holds for $G$ if and only if it holds for each connected component of $G$.
Let $H$ be the graph obtained from $G'$ by deleting $\lfloor\frac{d}{2}\rfloor$ loops from every vertex. Then $H$ is $\lceil\frac{d}{2}\rceil$-sparse and hence the minimum degree of $H$ is at most $d+1$. Furthermore, if the minimum degree of $H$ is equal to $d+1$, then $d$ is odd, $H$ is $(d+1)$-regular and $P=\emptyset$. 

Let $v$ be a vertex of minimum degree in $H$. Since each vertex in $P$ has degree at least $d+1$ in $H$,
$v\in V-P$.
Then $(G-v)'=G'-v$ satisfies the hypotheses of the lemma and hence $(G-v,P)$ is pinned independent in $\mathbb{R}^d$ by induction. If $d_H(v)\leq d$ , then $G$ can obtained from $G-v$ by a 0-extension and Lemma \ref{lem:01ext}  implies that  $(G,P)$ is pinned independent. Hence we may suppose that $d_H(v)=d+1$. As noted above, this implies that  $d$ is odd, $P=\emptyset$, $H=G$ and $G$ is $(d+1)$-regular.  

We will show that $G-v+xy$ satisfies the hypotheses of the lemma for two non-adjacent neighbours $x,y$ of $v$ in $G$. Since $P=\emptyset$, this is equivalent to showing that $G-v+xy$ is $\frac{d+1}{2}$-sparse and has no $K_{d+2}$-subgraph.
Since $d$ is odd, $G\neq K_{d+2}$ and we may choose $x,y\in N(v)$ such that $xy\notin E$. Since $G$ is connected and $(d+1)$-regular, we have $i(X)<\frac{d+1}{2}|X|$ for all $X\subsetneq V$ and hence $G-v+xy$ is $\frac{d+1}{2}$-sparse.
Suppose 
$G-v+xy$ contains a subgraph $K$ isomorphic to $K_{d+2}$.  Then $x,y\in V(K)$, and the fact that $G$ is $(d+1)$-regular implies that $N(v)\cap V(K)=\{x,y\}$. We can now deduce that, for all  $z\in N(v)-\{x,y\}$, $zx\not\in E$ and $G-v+xz$ is $\frac{d+1}{2}$-sparse and has no $K_{d+2}$-subgraph. 

By induction $G-v+xz$ is (pinned) independent in $\mathbb{R}^d$. Since $G$ is obtained from $G-v+xz$ by a 1-extension,  $G$ is (pinned) independent by Lemma \ref{lem:01ext}. This completes the proof.
\end{proof}

\section{Linearly constrained rigidity}

Let $(K_n^{[d+1]}, p, q)$ be a generic $d$-dimensional realization of the complete graph on $n$ vertices with $d+1$ loops on each vertex.
Since each edge/loop of $K_n^{[d+1]}$ is associated with a row of $R(K_n^{[d+1]}, p, q)$, we can define 
a matroid  on the union of the edge set and the  loop set of $K_n^{[d+1]}$ by the  linear independence of the row vectors of $R(K_n^{[d+1]}, p, q)$. This matroid is called the {\em generic linearly constrained rigidity matroid} $\mathcal{R}_{d,n}$.
A looped simple graph $G=(V,E,L)$ 
with $n$ vertices 
is said to be an {\em $\mathcal{R}_d$-circuit} if $E\cup L$ is a circuit in $\mathcal{R}_{d,n}$. 

We first derive a rather surprising result concerning the  infinitesimal motions of an arbitrary linearly constrained framework in $\R^d$.

\begin{lem}\label{lem:shincircuit}
Let $(G,p,q)$ be a generic linearly constrained framework in $\R^d$. Suppose that $v$ is a vertex of $G$ and $\rank R(G,p,q)=\rank R(G-\ell,p,q)$ for some loop $\ell$ incident to $v$.
Then $\dot p(v)=0$ for every infinitesimal motion $\dot p$ of $(G,p,q)$. 
\end{lem}

\begin{proof}
We proceed by induction on $|E(G)|$. 
The hypothesis that $\rank(G,p,q)=\rank(G-\ell,p,q)$ implies that $\ell$ is contained in some ${\mathcal R}_d$-circuit $C$ in $G$. 
If $C\neq G$ then we can apply induction to $(C,p|_{V(C)},q|_{V(C)})$ to deduce that $\dot p(v)=0$. Hence  we may suppose $G=C$. If $v$ is incident with $d$ loops then $v$ is fixed in every infinitesimal motion of  $(G,p,q)$ so we may suppose $v$ is incident to at most $d-1$ loops. Since $G$ is a ${\mathcal R}_d$-circuit, this implies that $v$ is incident to an edge $e\in E(G)$.

Let $G^+$ be the looped simple graph obtained from $G$ by adding a new loop $\ell^*$ at $v$ and put $G^*=G^+-\ell$. Then $G$ and $G^*$ are isomorphic so are both $\mathcal{R}_d$-circuits in the linearly constrained rigidity matroid of $G^+$. Since $e$ is a common edge of $G$ and $G^*$, we can apply the matroid circuit  exchange axiom to deduce that there exists a third $\mathcal{R}_d$-circuit $G'\subseteq G^+-e$. Since $G$ and $G^*$ are $\mathcal{R}_d$-circuits,  $\ell$ and $\ell^*$ are both loops in $G'$. 
Since $|E(G')<|E(G)|$, we can apply induction to deduce that $v$ is fixed in every infinitesimal motion of any generic realisation $(G',p',q')$ of $G'$. Since $G'$ is a $\mathcal{R}_d$-circuit and $\ell^*\in L(G')$, the space of infinitesimal motions of $(G',p',q')$ and $(G'-\ell^*,p',q'|_{L(G')-\ell^*})$ are the same and hence $v$ is fixed in every infinitesimal motion of any generic realisation  of  $G'-\ell^*$. Since $G'-\ell^* \subseteq G$, the same conclusion holds for $G$.
\end{proof}

We next use Lemmas  \ref{lem:pin} and \ref{lem:shincircuit} to characterise independence  for generic linearly constrained frameworks in $\R^d$ when each vertex is incident with sufficiently many loops. We say that a looped simple graph is {\em $K_{k}$-free} if it has no subgraph isomorphic to the complete graph $K_k$.

\begin{thm}\label{thm:main}
Suppose $d\geq 2$ is an integer and $G=(V,E,L)$ is a looped simple graph  with the property  that every vertex of $G$  is incident with at least $\lfloor\frac{d}{2}\rfloor$ loops.  Then $G$ is independent in $\mathbb{R}^d$ if and only if 
$G$ is $d$-sparse and $K_{d+2}$-free.
\end{thm}

\begin{proof}
To prove the necessity we suppose that $G$ is independent in $\R^d$. Then every  subgraph $G'=(V',E',L')$ of $G$ is independent in $\R^d$. This implies that $|E'\cup L'|\leq d|V'|$ (since the rows of $R(G',p',q')$ are linearly independent for any  generic realisation $(G',p',q')$ of $G'$), and $G'\neq K_{d+2}$ (since $K_{d+2}$ is dependent as a bar-joint framework in $\mathbb{R}^d$). Hence $G$ is $d$-sparse and $K_{d+2}$-free.

To prove sufficiency, we suppose that $G$ is $d$-sparse and $K_{d+2}$-free.
We show that $G$ is independent in $\mathbb{R}^d$ by induction on $|V|+|E|$.
The  cases when $|V|=1$ and  when $G$ is disconnected are straightforward so we assume that $|V|\geq 2$ and $G$ is connected.

We next consider the case when $G$ has a $d$-tight proper subgraph $H$,
where  $H$ is said to be {\em proper} if it is  connected and $1<|V(H)|<|V|$. The hyothesis that $G$ is $d$-sparse implies that $H$ is an induced subgraph of $G$ so every vertex of $H$ is incident with at least $\lfloor\frac{d}{2}\rfloor$ loops. We can now use induction to deduce that $H$ is minimally rigid. Construct a new graph $G'$ from $G$ by replacing  $H$ with another $d$-tight subgraph $H'$ on the same vertex set which has  $d$ loops at each vertex and no edges. 
It is not difficult to see that replacing the $d$-tight subgraph $H$ by the $d$-tight subgraph $H'$ preserves $d$-sparsity  and does not create a copy of $K_{d+2}$, so $G'$ satisfies the hypotheses of the theorem. Furthermore, since $H$ is a connected graph on at least two vertices, $H$ contains at least one edge and hence we can apply induction to $G'$ to deduce that $G'$ is independent in $\mathbb{R}^d$.  
Since replacing the minimally rigid subgraph $H'$ of $G'$ by the minimally rigid subgraph $H$ will not change independence,  $G$ is independent in $\mathbb{R}^d$.  

It remains to consider the case when $G$ has no $d$-tight proper subgraph. We assume, for a contradiction, that $G$ is not independent in $\R^d$. 
Then $G$ has a subgraph $C$ which is an $\mathcal{R}_d$-circuit. 
We next use the same trick as in the proof of Lemma \ref{lem:shincircuit} to show that every vertex of $C$ which is incident to a loop in $C$ must be incident to $d$ loops in $G$. 

Suppose $v$ is 
incident to a loop $\ell$ in $C$, but 
not incident 
to $d$ loops in $G$.
Let $G^+$ be the looped simple graph obtained from $G$ by adding a new loop $\ell^*$ at $v$, 
and let 
$C^*=C-\ell+\ell^*$. We can show, as in the proof of Lemma \ref{lem:shincircuit}, that there exists an $\mathcal{R}_d$-circuit $C'\subset C\cup C^*\subset G^+$ which does not contain a given edge $e$ incident to $v$ in $C$. The assumptions that $G$ has no $d$-tight proper subgraph and $v$ is not incident to $d$ loops in $G$ imply that $G^+-e$ is $d$-sparse. Since $G^+-e$ is $K_{d+2}$-free, has at least $\lfloor\frac{d}{2}\rfloor$ loops at each vertex and  has fewer edges than $G$, we may apply induction to deduce that $G^+-e$ is independent in $\R^d$. This contradicts the fact that $G^+-e$ contains the  $\mathcal{R}_d$-circuit $C'$. Hence, every vertex of $C$ which is incident to a loop in $C$ must be incident to $d$ loops in $G$. 

Let $H=C-L(C)$ be the underlying simple graph of $C$ and $P$ be the set of all vertices in $C$ which are incident to at least one loop in $C$. Let $H'$ be obtained from $H$ by adding $d$ loops at each vertex in $P$ and $\lfloor\frac{d}{2}\rfloor$ loops at each vertex of $V(C)\setminus P$. Then $H'$ is a subgraph of $G$ so is $d$-sparse and $K_{d+2}$-free. 
We can now use Lemma \ref{lem:pin} to deduce that $(H,P)$ is pinned independent in $\mathbb{R}^d$ so $R^{pin}(H,P,p)$ has linearly independent rows for any generic $p$. We now compare $R^{pin}(H,P,p)$ and $R (C, p, q)$, for some generic $q$. Since $P$ is the set of  vertices of $C$ which are incident with a loop in $C$, $R (C, p, q)$ has the block structure 
$$
\kbordermatrix{
 & P  & V(C)\setminus P  \\
E(C)&* & R^{pin}(H,P,p)\\
L(C)&* & 0
 }.
$$
The row independence of $R^{pin}(H,P,p)$ now implies that every linear combination of the rows of $R (C, p, q)$ which sums to zero will have zero coefficients for the rows indexed by $E(C)$.
This contradicts the fact that $C$ is an $\mathcal{R}_d$-circuit.
\end{proof}

Theorem \ref{thm:main} can be restated as a characterisation of rigidity.

\begin{thm}\label{thm:new}
Suppose $d\geq 2$ is an integer and $G$ is a looped simple graph  with the property  that every vertex of $G$  is incident with at least $\lfloor\frac{d}{2}\rfloor$ loops.  Then $G$ is independent in $\mathbb{R}^d$ if and only if $G$ has a spannning, $d$-tight, $K_{d+2}$-free subgraph $H$ with the property that  every vertex of $H$  is incident with at least $\lfloor\frac{d}{2}\rfloor$ loops.
\end{thm}
\begin{proof}
If $G$ has a spanning subgraph $H$ with the properties listed  in the theorem then $H$ will be rigid in $\R^d$ by Theorem \ref{thm:main}.
On the other hand, if $G$ is rigid in $\R^d$, then we can choose an independent spanning subgraph $H'$ of $G$ with no edges and $\lfloor\frac{d}{2}\rfloor$ loops at each vertex, and then extend $H'$ to a minimally rigid spanning subgraph $H$ of $G$. Then $H$ will be independent and rigid so will be $d$-tight and $K_{d+2}$-free.
\end{proof}

\section{Open questions}

We close by mentioning two further problems for linearly constrained frameworks in $\R^d$ which are solved for $d=2$ but open when $d\geq 3$.
\\[1mm]
1. The characterisation of rigidity for generic linear constrained frameworks in $\R^2$ was extended in \cite{EJNSTW,KatTan} by allowing the linear constraints to be non-generic. 
It would be of interest to extend  Theorem \ref{thm:new} in the same way.
\\
2. A linearly constrained frameworks $(G,p,q)$ in $\R^d$ is {\em globally rigid} if it is the only realisation of $G$ in $\R^d$ which satisfies the same distance and linear constraints as $(G,p,q)$. It was  proved in \cite{GJN} that a linearly constrained framework $(G,p,q)$  in $\R^2$ is globally rigid if and only if every connected component $H$ of $G$ is either a single vertex with at least 2 loops, or is redundantly rigid (i.e. $H-f$ is rigid in $\R^2$ for all edges and loops $f$ of $H$) and `2-balanced'. It seems likely that this result can be extended to higher dimensional linearly constrained frameworks when each vertex is incident to sufficiently many linear constraints. More precisely we  conjecture that, if $(G,p,q)$ is a generic linearly constrained framework  with at least two vertices in $\R^d$ and every vertex of $G$  is incident with at least $\lfloor\frac{d}{2}\rfloor$ loops, then $(G,p,q)$ is globally rigid if and only if every connected component of $G$ is either a single vertex with at least $d$ loops, or is redundantly rigid in $\R^d$. (The necessary condition from \cite{GJN} that $G$ should be `$d$-balanced' follows from the assumption that every vertex of $G$  is incident with at least $\lfloor\frac{d}{2}\rfloor$ loops.)

\subsection*{Acknowledgements} We would like to thank the Heilbronn Institute for Mathematical Research for providing partial financial support for this research. The third author was supported by JST ERATO Grant Number JPMJER1903 and JSPS KAKENHI Grant Number JP18K11155.

\end{document}